\documentclass[12pt,a4paper]{article}

\usepackage[utf8]{inputenc}
\usepackage{lmodern}
\usepackage{microtype}
\usepackage{amsmath,amssymb,amsthm,mathtools}
\usepackage{geometry}
\usepackage{hyperref}
\usepackage{enumitem}

\hypersetup{colorlinks=false,linkcolor=blue,citecolor=blue,urlcolor=blue}

\usepackage{xparse}

\usepackage{booktabs,array,longtable}

\newtheorem{theorem}{Theorem}
\newtheorem{proposition}[theorem]{Proposition}

\theoremstyle{definition}
\newtheorem{definition}[theorem]{Definition}
\theoremstyle{remark}
\newtheorem{remark}[theorem]{Remark}
\newtheorem{example}[theorem]{Example}

\newcommand{\N}{\mathbb{N}}

\DeclareMathOperator{\fabop}{fab}

\newcommand{\fab}[1]{%
  \fabop\if\relax\detokenize{#1}\relax\else(#1)\fi%
}

\DeclareMathOperator{\fabfiveop}{fabfive}

\newcommand{\fabfive}[1]{%
  \fabfiveop\if\relax\detokenize{#1}\relax\else(#1)\fi%
}

\DeclareMathOperator{\gcdop}{gcd}

\title{A Divisor Parametrization for the Erd\H{o}s--Straus Conjecture}
\author{M. Bello-Hernández \and M. Benito  \and E. Fernández}
\date{\today}

\begin{document}
\maketitle

\centerline{\bf Abstract}
\noindent
We study representations of \(1/n\) as a sum of three unit fractions whose denominators are all divisible by a prescribed integer \(m\). After scaling, this is equivalent to representing \(m/n\) as a sum of three unit fractions. Our main focus is the Erd\H{o}s--Straus case \(m=4\). We introduce a divisor-based function \(\mathrm{fab}(n,a,b)\), prove that its admissible parameters recover exactly the decompositions of \(1/n\) with all three denominators divisible by \(4\), and compare this parametrization with well-known Type I/II descriptions.

We also relate the construction to a
shifted cubic equation and to the surface
\[
  P(u,v,w)=uvw-u-v,
\]
whose subfamily \(P(\alpha+1,4\beta+3,4\gamma+3)\) gives a natural source of
examples but contains no perfect squares. Finally, we prove a translation
invariance property of \(\fab{}\), derive a modular sieve, and report
computational evidence: all primes \(p\equiv1\pmod4\) with \(p<10^{14}\) are
detected by \(\fab{p,a,b}\) with \(1\le a,b\le 11\), although some composite
values require larger parameters. We conclude with comparisons to Bradford's
two-variable reduction and Ventas' FCT sources for the \(5/n\) setting.

\section{Introduction}

The Erd\H{o}s--Straus conjecture asserts that, for every integer $n\ge2$, there
exist positive integers $x,y,z$ such that
\begin{equation}\label{eq:ES}
        \frac4n=\frac1x+\frac1y+\frac1z.
\end{equation}
If the conjecture is known for a divisor $p$ of $n$, then a decomposition for
$4/p$ can be scaled to one for $4/n$.  Thus, as usual, it is enough to consider
prime values of $n$.  We refer to Mordell~\cite{Mor}, Yamamoto~\cite{Yam},
Vaughan~\cite{Vau}, Schinzel~\cite{Sch}, Elsholtz--Tao~\cite{ElsTao}, Bradford~\cite{Bra}, and our manuscript~\cite{BelBenFer} and the references therein  for
background on the conjecture and on other parametrizations of its solutions.

We present several complementary ways to generate and understand Erdős--Straus decompositions: (i) a general divisor-based procedure encoded by $\fab{}$, (ii) the well-known Type I/II parametrizations and a basic rigidity obstruction for finite congruence covers, (iii) an algebraic-geometric model via the shifted cubic equation and the surface $uvw-u-v=n$, and (iv) computational consequences, in particular a translation-invariant modular sieve.

We compare this divisor search with the well-known Type~I and Type~II forms.
It is easy to see that fixing three of the four parameters to lie in a finite set yields only finitely many congruence classes, which cannot cover all primes $p\equiv1\pmod4$.  This contrasts with the empirical effectiveness observed in the small search window of $1\leq a,b\leq11$ for the primes tested in our computations.

A second theme is the connection with the shifted equation
\[
        \frac1n=\frac1{n+a}+\frac1{n+b}+\frac1{n+c}
\]
and with the polynomial surface
\[
        P(u,v,w)=uvw-u-v.
\]
The decompositions produced by $\fab{}$ give rational points on this surface.
The integral subfamily $P(a+1,4b+3,4c+3)$ provides a simple source of examples,
but it also has a structural limitation: it contains no perfect squares.  This
is useful to keep in mind when comparing the polynomial parametrization with the
more flexible divisor search.

In addition, we show that the ideas developed here can be adapted to obtain decompositions of fractions of the form \(5/n\), suggesting that the method is not restricted to the classical case \(4/n\). We obtain the Bradford's two-variable reduction in \(5/n\) setting.

In the computational section we write down a 
modular lifting principle for bounded searches and explain how the residue-class
sieve is used in the computations.  It also clarifies why the bound
$1\leq a,b\leq 11$ should be interpreted cautiously: although it appears very
effective for the tested primes, there are composite values not detected  in that
window but detected as soon as one allows $a$ or $b$ to be slightly larger.

\section{A Divisor Identity and the Function $\fab{n,a,b}$}\label{sec:fab}

Let $n,k\in\N$ and put $A=n(n+k)$. Then
\begin{equation}\label{eq:divisor-identity}
        \frac1n=\frac1{n+k}+\frac{k}{A+d}+\frac{k}{A+d_1}
\end{equation}
is equivalent to
\(dd_1=A^2=n^2(n+k)^2.
\) See \cite{HuangVaugham}.

\begin{definition}\label{def:fab}
Let $n,a,b\in\N$. Define the function $\fab{n,a,b}$ to be the least positive divisor $k$ of $a+bn$ such that
\begin{equation}\label{eq:fab-conditions}
        k\equiv3\pmod4,
        \qquad
        4b\mid \frac{a+bn}{k}(n+k),
        \qquad
        4a\mid n\frac{a+bn}{k}(n+k),
\end{equation}
if such a divisor exists; otherwise set $\fab{n,a,b}=0$. We refer to $k,a,b$ satisfying \eqref{eq:fab-conditions} as \textit{admissible values} for $n$.
\end{definition}

\begin{proposition}\label{prop:fab}
If $\fab{n,a,b}=k>0$, then
\begin{equation}\label{eq:fab-decomp-one}
        \frac1n
        =\frac1{n+k}
        +\frac1{\dfrac{(a+bn)(n+k)}{bk}}
        +\frac1{\dfrac{(a+bn)(n+k)n}{ak}}.
\end{equation}
If moreover $n\equiv1\pmod4$, then \eqref{eq:fab-decomp-one} gives a solution of
\eqref{eq:ES}.
\end{proposition}

\begin{proof}
If $\fab{n,a,b}=k>0$, then 
\(        d=\frac ab(n+k)\) and $d_1=(b/a)n^2(n+k)$ are divisors of $A^2$.
Then \eqref{eq:divisor-identity} becomes
\begin{equation}\label{eq:special-identity}
        \frac1n
        =\frac1{n+k}
        +\frac{kb}{(n+k)(a+bn)}
        +\frac{ka}{n(n+k)(a+bn)}.
\end{equation} 
By \eqref{eq:fab-conditions}, the three denominators in \eqref{eq:fab-decomp-one} are
divisible by $4$.
\end{proof}

\begin{example}
For $n=5$, $\fab{5,1,1}=3$, and
\[
        \frac45=\frac12+\frac14+\frac1{20}.
\]
\end{example}

\begin{remark}[Prime squares]
The function $\fab{}$ also gives a very simple decomposition for the squares
of primes \(p\equiv3\pmod4\). Indeed, let
$        n=p^2,\,a=k=p,\, b=1.$
Then
\(\frac{a+bn}{k}=p+1,
\) \(n+k=p^2+p=p(p+1)\),  and \eqref{eq:fab-conditions} holds.
Hence
\(        \fab{p^2,p,1}>0.
\)

For primes \(p\equiv1\pmod4\), computations suggest that one can often find
a positive integer \(b\) such that \(\fab{p^2,p,b}>0\). This case appears to
be subtler: with \(n=p^2\), \(a=p\), one has
\[
        a+bn=p(bp+1),
\]
and one has to find a divisor \(k\equiv3\pmod4\) of \(bp+1\) satisfying the
remaining divisibility condition. We do not pursue this question here.
\end{remark}

\begin{theorem}[Completeness of the divisor identity]\label{teo:Completeness}
Let \(n\equiv1\pmod4\).  Suppose that
\begin{equation}
\label{eq:4div-decomp}
   \frac1n=\frac1x+\frac1y+\frac1z
\end{equation}
with \(4\mid x\), \(4\mid y\), \(4\mid z\).  Then there exist
\(a,b\in\mathbb N\) and an admissible divisor \(k\equiv3\pmod4\) such that
the divisor identity \eqref{eq:fab-decomp-one} associated with \(fab(n,a,b)\) reproduces exactly the
given decomposition.
\end{theorem}

\begin{proof}
Since all summands in \eqref{eq:4div-decomp} are positive, each denominator is
larger than $n$. In particular, $x>n$. Define
\begin{equation}
\label{eq:k}
   k=x-n.
\end{equation}
Since $4\mid x$ and $n\equiv 1\pmod 4$, we have
\(
   k=x-n\equiv 0-1\equiv 3\pmod 4.
\)
Moreover, $x=n+k$.

From \eqref{eq:4div-decomp}, we obtain
\begin{equation}\label{eq:z-relation-english}
   z(ky-nx)=nxy.
\end{equation}
Let
\begin{equation}
\label{eq:a_b}
   g=\gcdop(x,y),
   \qquad
   b=\frac{x}{g},
   \qquad
   a=\frac{ky-nx}{g}.
\end{equation}
The positivity of $a$ follows from \eqref{eq:z-relation-english}. Therefore $a,b\in\mathbb N$.

We now compute
\(  a+bn
   =ky/{g}.
\)
Since $g\mid y$, it follows that
\begin{equation}
\label{eq:div-k}
   k\mid a+bn.
\end{equation}
Write
\(
   q={(a+bn)}/{k}.
\)
Then \(   q={y}/{g}.\)
Consequently,
\(
   q(n+k)=\frac{y}{g}x=\frac{x}{g}y=by.
\)
Since $4\mid y$, we get
\begin{equation}\label{eq:first-admissibility-english}
   4b\mid q(n+k).
\end{equation}

On the other hand, dividing by $g$ in \eqref{eq:z-relation-english} and using the definition of $a$, we obtain
\[
az=\frac{nxy}{g}.
\]
Also
\({nxy}/{g}=nq(n+k).
\)
Thus
\(
   nq(n+k)=az.
\)
Since $4\mid z$, it follows that
\begin{equation}\label{eq:second-admissibility-english}
   4a\mid nq(n+k).
\end{equation}
Therefore, \eqref{eq:div-k}, \eqref{eq:first-admissibility-english}, and \eqref{eq:second-admissibility-english} say that  $k$ satisfies the admissibility conditions for the pair $(a,b)$.

It remains to check that the associated identity recovers the original triple
$(x,y,z)$. According to \eqref{eq:fab-decomp-one}, the first denominator is
\(
   n+k=x.
\)
The second denominator is
\[
   \frac{(a+bn)(n+k)}{bk}
   =
   \frac{(ky/g)x}{(x/g)k}
   =y.
\]
The third denominator is
\[
   \frac{(a+bn)(n+k)n}{ak}
   =
   \frac{(ky/g)xn}{ak}
   =
   \frac{nxy}{ag}.
\]
By \eqref{eq:z-relation-english} and $a=(ky-nx)/g$, we have $azg=nxy$, and hence
\[
   \frac{nxy}{ag}=z.
\]
Thus the identity \eqref{eq:fab-decomp-one} associated with $\fab{}$ is exactly \eqref{eq:4div-decomp}.
\end{proof}

\begin{remark} In the theorem above, we do not assert that $\fab{n,a,b}$ coincides with
the value of $k$ given by \eqref{eq:k}, with $a,b$ as in \eqref{eq:a_b}.
We only establish that $k\ge \fab{n,a,b}>0$ for those values of $a,b$.

\end{remark}

\section{Well-known Type I/II Forms and the Cubic Equation}\label{sec:type}

We shall only use the following standard parametrization as background;
see Mordell~\cite[p.~287]{Mor} or Yamamoto~\cite{Yam}.
In order to avoid a conflict with the parameters of the function
\(\fab{}\), we denote the parameters by capital letters
\(A,B,C,D\). Thus, for prime \(n\), a solution of \((1)\) exists if and
only if one can find positive integers \(A,B,C,D\) such that either
\begin{equation}
\label{eq:typeI-cond1}
   (4ABC-1)D=(A+B)n,
\end{equation}
or
\begin{equation}
\label{eq:typeII-cond1}
   (4ABC-1)D=An+B.
\end{equation}
Dividing by \(ABCDn\) gives, respectively,
\begin{equation}
   \frac4n
   =
   \frac1{ABCn}
   +
   \frac1{BCD}
   +
   \frac1{ACD},
\label{eq:typeI-cond}
\end{equation}
or
\begin{equation}
   \frac4n
   =
   \frac1{ABCn}
   +
   \frac1{BCD}
   +
   \frac1{ACDn}.
\label{eq:typeII-cond}
\end{equation}
Thus \eqref{eq:typeI-cond} has exactly one denominator divisible by \(n\), whereas
\eqref{eq:typeII-cond} has two.

\begin{remark}
To compare these formulae with the divisor identity associated with
\(\fab{}\) in Proposition~\ref{prop:fab}, one has to pass from decompositions of \(4/n\) to
decompositions of \(1/n\). Let \(a_{\mathrm f}\) and \(b_{\mathrm f}\) denote the two parameters of
the choice function \(\fab{n,a_{\mathrm f},b_{\mathrm f}}\), in order to
distinguish them from the parameters \(A,B,C,D\). The
completeness construction of Theorem~\ref{teo:Completeness} makes the comparison explicit.
For instance, put
\(  G=\gcd(An,D).
\)
If in either type we choose
\(   n+k=4ABCn,
\)
then
\(   k=n(4ABC-1),\,
   b_{\mathrm f}=\frac{An}{G}.
\)
Moreover, in type I decomposition $a_{\mathrm f}={Bn^2}/{G}$, while in type II decomposition $a_{\mathrm f}={Bn}/{G}$.

Likewise, by interchanging the order of the summands in the decompositions displayed in \eqref{eq:typeI-cond} and \eqref{eq:typeII-cond}, one obtains values of $k$, $a_{\mathrm{f}}$, and $b_{\mathrm{f}}$ for which the decompositions associated with these parameters, as prescribed by Proposition~\ref{prop:fab}, agree with the decompositions given there.

As noted above, these decompositions do not imply that $\fab{n,a_{\mathrm{f}},b_{\mathrm{f}}}=k$; they only show that
\[
0<\fab{n,a_{\mathrm{f}},b_{\mathrm{f}}}\le k.
\]
\end{remark}

The following elementary observation explains why fixing most of the
parameters in \eqref{eq:typeI-cond1} or \eqref{eq:typeII-cond1} cannot by itself lead to a finite congruence
cover for all primes \(p\equiv1\pmod4\). For a proof of this result, we refer the reader to \cite{BelBenFer}.

\begin{theorem}\label{thm:noresclasses}
Fix finite sets of triples of positive integers for the four possible choices of three fixed parameters. Then the corresponding finite-parameter subfamilies of either \eqref{eq:typeI-cond1} or \eqref{eq:typeII-cond1} cannot cover all primes \(p\equiv1\pmod4\).

\end{theorem}

\begin{remark}\label{rem:contrast}
This obstruction contrasts with the computational behaviour of $\fab{}$: the authors
have verified that, for every tested prime
\[
        5\le p\equiv1\pmod4,
        \qquad p<10^{14},
\]
there exist $1\le a,b\le 11$ such that $\fab{p,a,b}>0$.
\end{remark}


The shifted decomposition
\begin{equation}\label{eq:shifted}
        \frac1n=\frac1{n+a}+\frac1{n+b}+\frac1{n+c}
\end{equation}
is equivalent, after clearing denominators, to
\[
        abc=n^2(2n+a+b+c),
\]
or
\[
        2n^3+(a+b+c)n^2-abc=0.
\]

\begin{remark}
The change of variables induced by the divisor identity in Proposition~\ref{prop:fab} separates the
geometric and arithmetic parts of the construction. Indeed, let
\(A,B,k,q\in \mathbb N\), and put
\[
   \alpha=k,\qquad
   \beta=\frac{A+nq}{B},\qquad
   \gamma=\frac{n^2(B+q)}{A}.
\]
Then a direct computation gives
\[
\alpha\beta\gamma
-
n^2(2n+\alpha+\beta+\gamma)
=
\frac{n^2}{AB}
(kq-A-Bn)(A+Bn+nq).
\]
Thus, in the positive range, the shifted cubic equation
\[
   \alpha\beta\gamma
   =
   n^2(2n+\alpha+\beta+\gamma)
\]
is equivalent, under this change of variables, to the linear relation
\[
   kq=A+Bn.
\]
This relation should not be confused with the admissibility conditions in
the definition of \(\fab{}\). It only ensures that the corresponding shifted
parameters lie on the cubic, possibly as rational points. To obtain an
integral decomposition one must also impose
\[
   4B\mid q(n+k),\qquad 4A\mid nq(n+k).
\]
Thus the divisor relation cuts out the cubic, whereas the remaining
admissibility conditions select the integral congruence-compatible points
produced by \(\fab{}\).
\end{remark}

A simple subfamily comes from
\begin{equation}
        P(u,v,w)=uvw-u-v.
        \label{eq:polManolo}
\end{equation}

If $P=P(u,v,w)>0$, then
\begin{equation}\label{eq:P-decomp}
        \frac1P
        =\frac1{P+v}
        +\frac1{u(wP+1)}
        +\frac1{P(wP+1)},
\end{equation}
and
\[
wP+1=(uw-1)(vw-1).
\]

In the shifted notation,
\[
        n=P,
        \qquad
        a=v,
        \qquad
        b=(uw-1)P+u,
        \qquad
        c=P^2w.
\]

For the congruence class relevant to \eqref{eq:ES}, define
\begin{equation}
\label{eq:PolyManolo}
 p(\alpha,\beta,\gamma)=P(\alpha+1,4\beta+3,4\gamma+3)
\end{equation}
that is,
\[
p(\alpha,\beta,\gamma)=(\alpha+1)(4\beta+3)(4\gamma+3)-(\alpha+1)-(4\beta+3).
\]

\begin{remark}
In \cite{BelBenFer} we show that the integral polynomial
subfamily \eqref{eq:PolyManolo}
can never account for square values of $n$.  This should be compared with the
 choice function $\fab{}$: as observed above, $\fab{}$ gives immediate
decompositions for $n=p^2$ when $p\equiv3\pmod4$, by taking $a=p$, $b=1$, and
$k=p$.  Thus the polynomial parametrization and the divisor search overlap, but
neither one should be viewed as a literal substitute for the other. Obviously, $k=4\gamma+3$, $a=\alpha+1$, and $b=1$ are admissible parameters for $n=p(\alpha,\beta,\gamma)$.
\end{remark}

\section{A $\fab{}$-Type Framework for the Sierpi\'nski--Schinzel Case}

The $\fab{}$ construction for the Erd\H{o}s--Straus equation is naturally
formulated in terms of decompositions of $1/n$ whose denominators are all divisible by $4$.
From this point of view, the appropriate analogue for the numerator $5$ problem is obtained
by considering decompositions
\[
\frac{1}{n}=\frac{1}{X}+\frac{1}{Y}+\frac{1}{Z},
\qquad 5\mid X,\quad 5\mid Y,\quad 5\mid Z,
\]
since, after writing
\[
X=5x,\qquad Y=5y,\qquad Z=5z,
\]
one immediately recovers an Egyptian decomposition of
\[
\frac{5}{n}=\frac{1}{x}+\frac{1}{y}+\frac{1}{z}.
\]
Thus, the numerator $5$ case can be studied through a parametrization of those
decompositions of $1/n$ in which the three denominators are multiples of $5$.

As in the Erd\H{o}s--Straus situation, the residue class that deserves special attention is the one
for which no uniform elementary decomposition is available. In the present setting, the cases
$n\equiv 0,2,3,4\pmod 5$ admit straightforward decompositions into three unit fractions,
whereas the genuinely difficult class is
\[
n\equiv 1 \pmod 5.
\]
This is the exact analogue of the residue class $n\equiv 1\pmod 4$ in the
Erd\H{o}s--Straus conjecture, and it is therefore the natural setting in which to introduce
a $\mathrm{fab}$-type device.

For $n\equiv 1\pmod 5$, the first denominator in the relevant identity will be of the form
$n+k$. In order to force this denominator to be divisible by $5$, one must impose
\[
k\equiv 4\pmod 5.
\]
This leads to the following definition.

\begin{definition}
Let $n\equiv 1\pmod 5$, and let $a,b\in\N$. We define $\fabfive{n,a,b}$ to be
the small divisor $k$ of $a+bn$ such that
\[
k\equiv 4\pmod 5,
\qquad
5b \mid \frac{a+bn}{k}(n+k),
\qquad
5a \mid n\,\frac{a+bn}{k}(n+k),
\]
provided such a divisor exists. If no such divisor exists, we set
\[
\fabfive{n,a,b}=0.
\]
\end{definition}

Viewing the problem as one of representing the fraction as a sum of three unit fractions shows that many of the results established for decompositions with denominators divisible by \(4\) can be extended without difficulty to denominators divisible by other values. In particular, in the case of denominators divisible by \(5\), the analogues of Proposition~\ref{prop:fab} and Theorem~\ref{teo:Completeness} are the following results which can be proved identically to the corresponding results.

\begin{proposition}\label{prop:5dec}
Assume that $n\equiv 1\pmod 5$ and that $\fabfive{n,a,b}=k>0$. If
\[
q=\frac{a+bn}{k},
\]
then
\begin{equation}
\label{eq:fab5-decomp-one}
\frac{1}{n}
=
\frac{1}{n+k}
+
\frac{1}{\,q(n+k)/b\,}
+
\frac{1}{\,nq(n+k)/a\,}.
\end{equation}
Moreover, each of the three denominators
\[
n+k,\qquad \frac{q(n+k)}{b},\qquad \frac{nq(n+k)}{a}
\]
is divisible by $5$. Consequently, one obtains a decomposition of
\(\frac{5}{n}
\)
into three unit fractions.
\end{proposition}

\begin{theorem}[Completeness of $\fabfive{}$]\label{teo:complete5}
Let $n\equiv 1\pmod 5$, and suppose that
\[
\frac{1}{n}=\frac{1}{X}+\frac{1}{Y}+\frac{1}{Z},
\qquad 5\mid X,\quad 5\mid Y,\quad 5\mid Z.
\]
Then there exist
\(a,b\in\mathbb N\) and an admissible divisor \(k\equiv4\pmod5\) such that
the divisor identity \eqref{eq:fab5-decomp-one} associated with \(\fabfive{n,a,b}\) reproduces exactly the
given decomposition.
\end{theorem}

\begin{remark}[The numerator~5 polynomial family]
It may be useful to record that the polynomial surface \eqref{eq:polManolo}
behaves differently in the numerator~5 setting from the corresponding
numerator 4 subfamily.  The analogue  leads instead to the congruence condition
\[
  v\equiv w\equiv 4 \pmod 5.
\]
This family has no analogous square obstruction.  For instance,
\[
  P(3,14,9)=3\cdot14\cdot9-3-14=361=19^2.
\]
Thus the polynomial family with v and w congruent to 4 modulo 5 already
contains prime-square values.

Computationally, this family also appears to be very effective for primes in
the difficult numerator 5 residue class.  A finite sieve based on the identity
\[
  n+v=u(vw-1)
\]
shows that, up to the tested bound, the primes congruent to 1 modulo 5 are
almost always represented by
\[
  n=P(u,v,w),
  \qquad v\equiv w\equiv4\pmod5.
\]
In the range tested up to $10^6,$ the only missing primes are
\(  541,\,1381.
\)
This should be contrasted with the numerator~4 polynomial subfamily, where
square values are excluded for structural reasons.  The numerator~5 polynomial
family therefore seems worth mentioning, not as a replacement for the
$\fabfive$ framework, but as an additional comparison showing that the behaviour of
$P(u,v,w)$ depends strongly on the chosen congruence class. Observe that $\fabfive{541,1,2}=19$ and $\fabfive{1381,1,2}=9$, which yields
\[
  \frac{5}{541}  =  \frac{1}{3453744}  +  \frac{1}{3192}  +  \frac{1}{112},
\]
and 
\[
  \frac{5}{1381}
  =
  \frac{1}{117862826}  +  \frac{1}{42673}  +  \frac{1}{278}.
\]

Of course, $k=5\gamma+4$, $a=\alpha+1$, and $b=1$ are admissible parameters for $n=P(\alpha+1,5\beta+4,5\gamma+4)$ according to the function $\fabfive{n,a,b}$.
\end{remark}

Finally, the same parametrizing argument extends without essential change to a
general numerator \(m\geq 2\), provided one interprets the problem in terms of
decompositions of \(1/n\) whose three denominators are all divisible by \(m\).
In this setting, the congruence condition becomes
\[
k\equiv -n\pmod m,
\]
and the divisibility conditions are replaced by
\[
mb \mid \frac{a+bn}{k}(n+k),
\qquad
ma \mid n\,\frac{a+bn}{k}(n+k).
\]
Whenever such an admissible divisor \(k\) exists, the associated identity
produces a decomposition of \(1/n\) with all three denominators divisible by
\(m\), and hence, after scaling, a decomposition of \(m/n\) into three unit
fractions.

Thus the cases \(m=4,5\) are  particular instances of a more general
\(\mathrm{fab}_m\) parametrization. However, the assertion that the only
residue class requiring special attention is \(n\equiv 1\pmod m\) should not be
made for arbitrary \(m\). For \(m=4\) and \(m=5\) the complementary residue
classes are covered by elementary identities, and the same phenomenon also
occurs for \(m=6\). For larger numerators, however, the elementary treatment of
the complementary classes is no longer automatic: after subtracting the natural
first unit fraction one is led to a two-term Egyptian decomposition of
\(r/N\), with \(r\) possibly larger than \(4\), and this imposes additional
arithmetic conditions.

\section{Further Comparisons}

\label{sec:related-approaches}

We finish by recording two links between the divisor identity used in this paper
and related approaches to the Erd\H{o}s--Straus conjecture. The purpose of this
section is not to give a survey, but rather to clarify how our parametrization
interacts with Bradford's two-variable reduction and with Ventas' source-based
FCT construction.

\subsection*{Bradford's Reduction}

Bradford (\cite{Bra} and \cite{Bra2025}) observed that, once a solution is given, the third denominator is forced by the first two through a gcd identity: if \(p\equiv 1\) (mod 4) is a prime,  
\[
  \frac4p=\frac1x+\frac1y+\frac1z,
  \qquad x\le y\le z,
\]
then 
\begin{equation}
\label{eq:BradfordId}
z=\frac{xyp}{\gcd(y,p)\gcd(xy,x+y)}.
\end{equation}
By Theorem~\ref{teo:Completeness}, we have
\[
z=\frac{pxy}{ag},\quad a=\frac{ky-px}{g},
\]
where $g=\gcd(x,y)$. Then from Bradford's identity \eqref{eq:BradfordId} we obtain \begin{equation}
\label{eq:ParamA_Bradford}
  a=\frac{\gcd(y,p)\gcd(xy,x+y)}{\gcd(x,y)}
\end{equation}
and 
\begin{equation}
\label{eq:ParamD_Bradford}
  D=\gcd(y,p)\gcd(xy,x+y),
\end{equation}

In view of the symmetry between the formulations for \(4/n\) and \(5/n\) (see Proposition \ref{prop:5dec} and Theorem \ref{teo:complete5}), it is natural to expect the following analogue.
\begin{theorem}[Bradford's defect and the canonical $\fabfive{}$ parameters]
\label{lem:bradford-defect-fabfive-compact}
Let $p\equiv1\pmod5$ be prime, and let
\begin{equation}
\label{eq:Erdos-Straus5}
  \frac5p=\frac1x+\frac1y+\frac1z,
  \qquad x\le y\le z,
\end{equation}
be a solution. Put
\(  D_5=5xy-p(x+y).
\)
Then 
\begin{equation}
\label{eq:BradfordId5}
z=\frac{xyp}{\gcd(y,p)\gcd(xy,x+y)}.
\end{equation}
Therefore
\begin{equation}
\label{eq:ParamD_Bradford5}
  D_5=\gcd(y,p)\gcd(xy,x+y),
\end{equation}
and the parameter $a$ in the $\fabfive{}$ function  is given by
\begin{equation}
\label{eq:ParamA_Bradford5}
  a=\frac{\gcd(y,p)\gcd(xy,x+y)}{\gcd(x,y)}.
\end{equation}
\end{theorem}

\begin{proof}
By \eqref{eq:Erdos-Straus5} we have 
\begin{equation}
\label{eq:Erdos-StrausDesp}
D_5z=pxy.
\end{equation}
In particular, \(D_5\mid pxy\). Set \(\alpha=\gcd(D_5,p)\). As $p$ is a prime, $\alpha=1$ or $p$. If \(p\mid D_5\), as \(p>5\), we have \(p\mid x\) or \(p\mid y\). Since \(x\le y\le z\), the former condition and \eqref{eq:Erdos-Straus5} yield a contradiction. Thus, if \(p\mid D_5\), \(p\nmid x\) and \(p\mid y\). Therefore,
\[
\alpha=\gcd(D_5,p)=\gcd(y,p).
\]

As \(D_5\mid pxy\), we get
\begin{equation}
\label{eq:aux1}
\frac{D_5}{\alpha}\mid \gcd(xy,x+y).
\end{equation}

Let us prove the reverse divisibility. Obviously, 
\[
\gcd(xy,x+y)\mid D_5.
\]
We have proved that if \(\alpha=p\), then \(p\nmid x\) and \(p\mid y\). So $p\nmid x+y$ and $p\nmid \gcd(xy,x+y)\). Thus,
\[
\alpha\mid D_5
\]
Hence, we obtain
\begin{equation}
\label{eq:aux2}
\alpha \gcd(xy,x+y)\mid D_5.
\end{equation}
Therefore, combining \eqref{eq:aux1} and \eqref{eq:aux2} we get \eqref{eq:ParamD_Bradford5}
\[
D_5=\alpha \gcd(xy,x+y)=\gcd(y,p)\gcd(xy,x+y).
\]
This equation, together \eqref{eq:Erdos-StrausDesp} shows \eqref{eq:BradfordId5}.

The same argument as in the proof of Theorem~\ref{teo:Completeness}, applied to the \(\fabfive{}\) parametrization, gives
\[
z=\frac{pxy}{ag},
\]
and \eqref{eq:ParamA_Bradford5} follows.
\end{proof}

\subsection*{Ventas' FCT Sources}

Ventas' ceiling continued fraction construction (FCT, see \cite{Ven}) can also be interpreted inside the same divisor framework.  His source theorem assumes a divisor $d\equiv3\pmod4$ of an external source $p+i$, together with the divisibility condition $4i\mid p+d$.  This is exactly a sufficient condition for a certificate in the layer $b=1$ of our construction.

\begin{proposition}[FCT sources as $\fab{}$ certificates]
\label{prop:fct-as-fab-compact}
Let $p\equiv1\pmod4$ be prime.  Suppose that there exist $i,d\in\mathbb N$ such that
\[
  d\mid p+i,
  \qquad d\equiv3\pmod4,
  \qquad 4i\mid p+d.
\]
Then $d$ is an admissible divisor for $\fab{p,i,1}$.  If $q=(p+i)/d$, the corresponding identity is
\[
  \frac1p
  =\frac1{p+d}
   +\frac1{q(p+d)}
   +\frac1{pq(p+d)/i},
\]
and, after division of the three denominators by $4$, it gives
\[
  \frac4p=\frac1x+\frac1y+\frac1z,
\]
where
\[
  x=\frac{p+d}{4},
  \qquad
  y=\frac{(p+d)(p+i)}{4d},
  \qquad
  z=\frac{p(p+d)(p+i)}{4id}.
\]
\end{proposition}

\begin{proof}
The conditions $d\mid p+i$ and $d\equiv3\pmod4$ give the divisor and congruence requirements.  With $q=(p+i)/d$, the hypothesis $4i\mid p+d$ implies
\[
  4\mid q(p+d),
  \qquad
  4i\mid pq(p+d).
\]
These are precisely the remaining admissibility conditions for $\fab{p,i,1}>0$.  The displayed identity is the divisor identity for $a=i$, $b=1$, and $k=d$.
\end{proof}

\begin{remark}Thus Ventas' condition is naturally contained in the $b=1$ layer of $\fab{}$.  Conversely, $\fab{p,i,1}$ is slightly more flexible: it only requires
\[
  4\mid q(p+k),
  \qquad 4i\mid pq(p+k),
  \qquad q=\frac{p+i}{k},
\]
so the factor $q$ may also contribute to the required divisibility.

There is also a geometric overlap.  If an FCT solution has coefficients $\lceil c_0,c_1,c_2\rceil$ and final numerator $p$, the negative recurrence gives
\[
  p=(c_1c_2-1)c_0-c_2=c_0c_1c_2-c_0-c_2.
\]
Hence, for $P(u,v,w)=uvw-u-v$,
\[
  p=P(c_2,c_0,c_1),
\]
so the FCT grid lies on the same cubic surface, up to a permutation of variables.
\end{remark}

These two comparisons indicate that the divisor identity is not merely
another parametrization, but a convenient common language in which several
apparently different constructions can be expressed.

\begin{remark}
We observe that the preceding result carries over without difficulty to the
setting of Sierpiński's conjecture. Indeed, for those values
\(n\equiv 1 \pmod 5\) arising from the polynomial surface
\[
  P(u,v,w)=uvw-u-v,
\]
with \(v,w\equiv 4 \pmod 5\), the fractions \(5/n\) also admit
representations as sums of three unit fractions.
\end{remark}

\section{Translation Invariance, Modular Sieving, and Computation}
\label{sec:translation-sieving}

The choice function $\fab{}$ has a simple translation invariance that is useful in
modular sieving.

\begin{proposition}[Translation invariance]
\label{prop:translation}
Suppose $\fab{n,a,b}=k>0$ and put
\[
        n_1=n+4abk.
\]
Then the same $k$ remains an admissible divisor for the congruence and
divisibility tests defining $\fab{n_1,a,b}$.
\end{proposition}

\begin{proof}
Since $k\mid a+bn$, one has
\[
        a+bn_1=a+bn+4ab^2k,
\]
so $k\mid a+bn_1$. Moreover
\[
        \frac{a+bn_1}{k}=\frac{a+bn}{k}+4ab^2,
        \qquad
        n_1+k=n+k+4abk.
\]
Thus the defining congruence tests are preserved.
\end{proof}

As an immediate consequence of the translation invariance, one obtains the
following modular sieving algorithm, which is useful for bounded modular sieving.

\begin{remark}[Modular sieving]
Proposition~\ref{prop:translation} gives a practical way to organize the bounded
search.  For a fixed bound $C$, one may work modulo a convenient modulus $m$ and
remove a residue class $r\pmod m$ as soon as a certificate
\[
        \fab{r,a,b}=k>0,
        \qquad
        1\le a,b\le C,
        \qquad
        4abk\mid m
\]
is found.  The remaining residue classes are the only classes that still need to
be explored when looking for values not detected by the bounded window
$1\le a,b\le C$.  If $m_0\mid m$, survivor classes modulo $m_0$ can be lifted to
classes modulo $m$ by testing the finitely many residues
\[
        r+\ell m_0\pmod m,
        \qquad
        0\le \ell < m/m_0.
\]
This is the modular principle behind the computational sieve used below.
\end{remark}

\begin{remark}[Modular inverse role]
Suppose that \(a\) and \(b\) are coprime to \(k\equiv 3\pmod 4\). Let $\overline{a}$ and $\overline{b}$ be such that $a\overline{a}\equiv 1$ and $b\overline{b}\equiv 1$ (mod $k$). The condition
\(k\mid a+bn\) is equivalent to
\[
n\equiv -a\overline{b}\pmod k,
\]
that is, \(k\mid \overline{b}+\overline{a}n\). Thus, when considering the divisors
\(k\equiv 3\pmod 4\) of \(a+bn\), it is not only natural to ask whether
\eqref{eq:fab-conditions} holds for such a \(k\), but also whether
\begin{equation*}
        4\overline{a}\mid \frac{\overline{b}+\overline{a}n}{k}(n+k),
        \qquad
        4\overline{b}\mid n\frac{\overline{b}+\overline{a}n}{k}(n+k).
\end{equation*}

If \(k\mid a+bn\), but either
\[
4a\nmid \frac{a+bn}{k}(n+k)
\qquad\text{or}\qquad
4b\nmid \frac{a+bn}{k}(n+k)n,
\]
then one may try to shift the values of \(a\) and \(b\) by suitable multiples of \(k\), while preserving the congruence \(k\mid a+bn\), in order to make both divisibility conditions hold. For large \(n\), this approach appears to provide a promising strategy for finding admissible triples \((a,b,k)\) associated with a given value of \(n\).
\end{remark}

\begin{remark}[Composite exceptional values for the bounded search]\label{re:compExc4} The 
bounded search for composite \(n\equiv 1 \pmod 4\) analogous to Remark \ref{rem:contrast} behaves differently.
If one restricts Definition~1 to the window \(1\leq a,b\leq 11\), then
some composite values are not detected in that window, although they are
detected as soon as one allows one of the parameters to be slightly larger.
The only numbers, which are not square, that require $a$ or $b$ greater than 11 up to $10^{14}$ appear in Table \ref{TableNoWindow4}. Here \(\operatorname{comab}(n)=(a,b,k)\) means that \(k,a,b\) are  admissible
values for $n$.

\begin{table}[ht]
\centering
\resizebox{0.7\textwidth}{!}{%
\begin{tabular}{c c c }
\toprule
$n$ & small factor of $n$ & $\operatorname{comab}(n)$\\
\midrule
$68889266161$ & $43969$ & $(12,1,2639)$ \\
$198670395169$ & $38791$ & $(1,13,1171)$ \\
$2081413401769$ & $13$ & $(5,12,189415991)$ \\
$32140320080401$ & $13$ & $(13,1,7)$ \\
$42042659579881$ & $1699$ & $(7,15,770953415279)$ \\
$58797131752129$ & $19$ & $(7, 12, 25015620671)$ \\
$87574118340241$ & $13$ & $(13, 1, 7)$\\
\bottomrule
\end{tabular}
}
\caption{This table lists the only nonsquare integers \(n \le 10^{14}\)
for which \(\fab{n,a,b}\) requires parameters \(a,b\) outside the window
\(\max\{a,b\}\le 11\).} \label{TableNoWindow4}
\end{table}

These examples do not contradict the prime computation, since they are
composite rather than prime. They suggest, however, that the small window
\(1\leq a,b\leq 11\) should be regarded as a phenomenon of the prime search,
not as a uniform bound for all composite integers.
\end{remark}

\begin{remark}[Greedy survivor sequences] \label{re:PrimeStairs}
The translation-invariant sieve gives rise to natural greedy survivor
sequences.  Starting from the single sieve value $k=3$, one repeatedly takes
the least prime, or composite, not detected by the current finite sieve and
then adjoins the least new admissible $k\equiv3\pmod4$ which detects it.  The
first prime survivors are
\[
73,1129,1201,3361,5569,9241,14401,\ldots,
\]
whereas the first composite survivors, with the prime-square rule $a=p$ for
$n=p^2$, are
\[
49,1369,1849,2641,5161,6241,11089,\ldots .
\]
The largest prime survivor below \(10^7\) found in our computations is
\(8803369\), which requires \(k=107\). These sequences seem to encode the early growth of the modular sieve associated
with $\fab{}$.
\end{remark}

\begin{remark}\label{re:compExc5}Using the choice function \(\fabfive{n,a,b}\) to decompose \(\frac{5}{n}\),
we work within the window \(\max\{a,b\}\le 9\). The only integers
up to \( 10^{14}\) that require values outside this window are in Table \ref{TableNoWindow5}.

\begin{table}[ht]
\centering
\resizebox{0.6\textwidth}{!}{%
\begin{tabular}{c c c }
\toprule
$n$ & smallest factor of $n$ & $(a,b)$\\
\midrule
$305945641$ & $n$ (\text{prime}) & $(10,1)$ \\
$2965123604521$ & $13$ & $(2,11)$ \\
$7171425327781$ & $11$ & $(11,1)$ \\
$2095120616401$ & $739$ & $(10,3)$ \\
$4269339137701$ & $11$ & $(11,1)$ \\
$11471606546401$ & $37$ & $(10,1)$ \\
$45931894495201$ & $11$ & $(1,10)$\\
$63712786956841$ & $11$ & $(1,10)$\\
$80611085041201$ & $11$ & $(10,3)$\\
\bottomrule
\end{tabular}
}
\caption{This table lists the only integers \(n \le 10^{14}\) for which
\(\fabfive{n,a,b}\) requires a choice of parameters \(a,b\) outside the
window \(\max\{a,b\}\le 9\).} \label{TableNoWindow5}
\end{table}

\end{remark}

All computations reported in Remarks \ref{rem:contrast}, \ref{re:compExc4}, \ref{re:PrimeStairs}, and \ref{re:compExc5} were performed
by an independent implementation of the admissibility conditions in
Definition \ref{def:fab}; the code is available from the authors upon request.

\section*{Acknowledgments and AI Tool Disclosure.}

Research of M. B-H. is supported in part by the grant PID2022--138342NB--I00 AEI, from Spanish Government.

ChatGPT 5.5 was used for proofreading. Outside of this AI tool use, the mathematics and the text of this paper were both human-generated.

\newpage


\noindent
\textbf{M. Bello-Hernández}, corresponding author. Departamento de Matemáticas y Computación, Universidad de La Rioja. c/ Madre de Dios, 53, 26006 Logroño, La Rioja, Spain. Research supported in part by the grant PID2022--138342NB--I00 AEI, from Spanish Government. \\\texttt{mbello@unirioja.es}

\vspace{0.3cm}

\noindent
\textbf{M. Benito}, retired Professor. Instituto P. M. Sagasta, Logro{\~n}o and Departamento de Matemáticas y Computación, Universidad de La Rioja. c/ Madre de Dios, 53, 26006 Logroño, La Rioja, Spain. \\ \texttt{mbenitomunnoz@gmail.com}

\vspace{0.3cm}

\noindent
\textbf{E. Fernández}, retired Professor. Instituto P. M. Sagasta, Logro{\~n}o and Departamento de Matemáticas y Computación, Universidad de La Rioja. c/ Madre de Dios, 53, 26006 Logroño, La Rioja, Spain. \\\texttt{emilio.fernandez@unirioja.es}

\end{document}